\documentclass[12pt,a4paper,oneside,titlepage]{article}
\usepackage{amsmath,amsfonts,amssymb,amsthm,amscd,verbatim,hyperref}
\theoremstyle{plain}
\newtheorem{Theorem}{Theorem}[section]

\newtheorem{Lemma}[Theorem]{Lemma}
\newtheorem{Corollary}[Theorem]{Corollary}

\theoremstyle{definition}

\newtheorem{Remark}[Theorem]{Remark}
\newtheorem{Example}[Theorem]{Example}

\newcommand{\R}{\mathbb R}
\newcommand{\Q}{\mathbb Q}
\newcommand{\Z}{\mathbb Z}
\newcommand{\N}{\mathbb N}

\newcommand{\alaplus}{\genfrac{}{}{0pt}{}{}{+}}
\newcommand{\aladots}{\genfrac{}{}{0pt}{}{}{\cdots}}
\newcommand{\K}{\mathop{\mathchoice{\doK\LARGE}{\doK\Large}{\doK\small}{\doK\small}}}
\newcommand{\doK}[1]{\vcenter{#1\kern.2ex\hbox{\normalfont\text{K}}\kern.2ex}}

\begin{document}

\title{On irrationality exponents of generalized continued fractions}
\author{Jaroslav Han\v cl \footnote{The work of Jaroslav Han\v cl was supported by the European Regional Development Fund in the IT4Innovations Centre of Excellence project (CZ.1.05/1.1.00/02.0070) and by grant 
no. P201/12/2351.},  Kalle Lepp\"al\"a, Tapani Matala-aho and Topi T\"orm\"a
\footnote{The work of all authors was supported by the Academy of Finland, grant 138522.}}
\maketitle
\date

\begin{abstract}
We study how the asymptotic irrationality exponent of a given generalized continued fraction 
\[
\K_{n=1}^\infty \frac{a_n}{b_n}\,,\quad a_n, b_n\in \mathbb{Z}^+,
\]
behaves as a function of growth properties of partial coefficient sequences $(a_n)$ and $(b_n)$.
\end{abstract}

\section{Introduction}

Our target is to present upper bounds for asymptotic irrationality exponents of
generalized continued fractions with positive integer partial coefficients.

By the concept of irrationality exponent of a real number $\tau$ we mean any exponent $\mu$ for which
there exist positive constants $c$ and $N_0$ such that
\begin{equation}\label{mudefinition}
\left| \tau-\frac{M}{N} \right| \ge \frac{c}{N^\mu}
\end{equation}
holds for all $M, N \in \Z$, $N \ge N_0$. 
The asymptotic irrationality exponent $\mu_I(\tau)$ is then the infimum of all such exponents $\mu$. 
Here the infimum of an empty set is interpreted as $\infty$,
in which case $\tau$ is called a Liouville number. For irrational $\tau$ we have $\mu_I(\tau) \ge 2$ by Dirichlet's theorem on Diophantine approximation and for rational $\tau$ holds $\mu_I(\tau)=1$.
Any function (not necessarily a power function) bounding $\left| \tau-\frac{M}{N}\right|$ from below for big enough $N$ is in turn called an irrationality measure of $\tau$.

By generalized continued fraction we mean the expression
\begin{equation}\label{GC}
b_0+\K_{n=1}^{\infty} \frac{a_n}{b_n} = b_0+\frac{a_1}{b_1+\dfrac{a_2}{b_2+\cdots}}\,,
\end{equation}
as well as the value of the limit when it exists.
The numbers $a_n$ and  $b_n$ are called the partial numerators and denominators, respectively,
together referred as the partial coefficients.
In this paper we assume that the partial coefficients are positive integers
bounded either by constants, by polynomials in $n$ or exponentially in $n$.

As far as we know there doesn't exist a systematic approach for studying the irrationality
exponents of generalized continued fractions. The purpose of this work is to (partially) solve this problem 
as an analogy to the approach in \cite{hanleilepmat} devoted to simple continued fractions.

Contrary to the case of simple continued fractions, generalized continued fractions are not always irrational and in fact don't even necessarily converge.
There exist numerous criteria for detecting the convergence of a generalized continued fraction, 
see Lorentzen and Waadeland \cite{lorentzen} and Perron \cite{perron}.
For irrationality criterion we mention Tietze's Criterion (see Perron \cite{perron}, s.242, Satz 14)
which tells that if $1\le a_n \le b_n$, $a_n,b_n \in \Z^+$,
then the continued fraction \eqref{GC} converges to an irrational number.
Bowman and Mc Laughlin \cite{bowman} considered generalized continued fractions 
that have values of polynomials as partial coefficients, with particular attention given to cases 
in which the degrees of the polynomials are equal.
As an example of a rational limit from \cite{bowman} we give
\begin{equation}\label{rationalexample}
\K_{n=1}^\infty\frac{6n^7 + 6n^6 + 2n^5 + 3n + 2}{6n^7 - 6n^6 + 2n^5 + 3n- 5}=\frac{19}{7}\,.
\end{equation}
However, in \cite{bowman} no irrationality measure results are given.

We only know few papers with irrationality measure results for ge\-ne\-ra\-li\-zed continued fractions. 
In \cite{shiokawa} Shiokawa applied a generalized continued fraction expansion of the exponential function
in studying irrationality measures of the exponential function in non-zero rational points.
For $\pi$, none of the several generalized continued fraction expansions is suitable for pro\-ving irrationality measure results.
The reason is that the growth rate of the partial numerators is too high compared to the growth rate of the partial denominators. 
This is also the case for most of the other mathematical constants. 

We start by proving Lemma \ref{INVARIANT} which shows that any irrationality exponent is invariant under linear fractional transformations. 
This means that in our purposes it is enough to consider the tails of generalized continued fractions.
Then we present a general result Lemma \ref{Theorem3.1} which gives an upper bound for the asymptotic irrationality exponent when we know some growth pro\-per\-ties of the numerator and denominator sequences of the convergents.
Lemmas \ref{INVARIANT} and \ref{Theorem3.1} are variations of generally known methods, 
see for example \cite{feldman} and \cite{steuding}.

We continue by applying Lemma \ref{Theorem3.1} in three cases.

In the first case the sequences of partial coefficients are bounded by constants.
In Theorem \ref{Theorem3.2} we bound the asymptotic irrationality exponent with a function of these constants, provided that they satisfy certain condition \eqref{gamma1condition}.
We show that condition \eqref{gamma1condition} is not unnecessary: 
without such restriction we can construct generalized continued fractions with arbitrarily high asymptotic irrationality exponents (or even Liouville numbers) using only values $1$ and $2$.
As examples in the constant category we study numbers connected to Thue-Morse and Fibonacci binary sequences.
These sequences have attracted a great deal of attention recently, see e.g. \cite{adabug2007} and \cite{adabug2005}.

In the second case the sequences of partial coefficients are bounded by polynomials. 
As an example we study a generalization of the simple continued fraction
\begin{equation}\label{bundcon}
\tau=[0,c_1+d_11^m,\ldots,c_s+d_s1^m,c_1+d_12^m,\ldots,c_s+d_s2^m,\ldots]\,,
\end{equation}
where $m,s,c_i,d_i\in\Z^+$ for $i=1,2,\ldots,s$, which has been under intensive investigations 
from the times of Euler \cite{euler}, see Lehmer \cite{lehmer1973} and Bundschuh \cite{bundschuh1998}
for references. We relax the conditions by allowing the numbers $c_i$ and $d_i$ to be positive rationals
and then use a certain equivalence transformation \eqref{equivtransform} to express 
$\tau$ in (\ref{bundcon}) as a generalized continued fraction and consequently to prove that
its asymptotic irrationality exponent is $2$.

In our final case the sequences of partial coefficients are bounded by exponential functions. 
As examples we consider certain $q$-fractions, including the Rogers-Ramanujan continued fraction and Tasoev's 
continued fractions.
Again, our results generalize or reproduce some earlier results, see e.g. Komatsu \cite{komatsu2003} and \cite{komatsu2005}. 
Further, we may study certain continued fractions generated by linear recurrences.

\section{Notations and preliminaries} 

Let $(a_n)$ and $(b_n)$ be sequences of positive integers (complex numbers in a more general definition).
The generalized continued fraction
\[
\tau = b_0+\frac{a_1}{b_1+\dfrac{a_2}{b_2+\cdots}} = b_0 + \K_{n=1}^\infty\frac{a_n}{b_n}=b_0+\frac{a_1}{b_1}\alaplus\frac{a_2}{b_2}\alaplus\aladots
\]
is defined as the limit of the $n$:th convergent
\[
\frac{A_n}{B_n}
=b_0+\K_{k=1}^n\frac{a_k}{b_k}=b_0+\frac{a_1}{b_1}\alaplus\frac{a_2}{b_2}\alaplus\aladots\alaplus\frac{a_n}{b_n}
\]
as $n$ tends to infinity. This limit does not necessarily exist, but does so in all the instances we study in this work.
The numerators $A_n$ and denominators $B_n$ of the convergents are integers satisfying the recurrence relations
\[
A_{n+2}=b_{n+2}A_{n+1}+a_{n+2}A_n \,,\qquad 
B_{n+2}=b_{n+2}B_{n+1}+a_{n+2}B_n
\]
with initial values $A_0=b_0$, $B_0=1$, $A_1=b_0b_1+a_1$ and $B_1=b_1$. 
By the $n$:th tail we mean the continued fraction
$\K_{k=n}^\infty \frac{a_k}{b_k}$.
From now on we will use the short-hand notations
$\Pi_n=\prod_{k=1}^{n}a_k$ and $R_n=\lvert B_n\tau-A_n \rvert$.

By the recurrence relations we get
\[
\frac{A_{k+1}}{B_{k+1}}-\frac{A_{k}}{B_{k}}=
\frac{(-1)^k\Pi_{k+1}}{B_{k}B_{k+1}}
\]
for all $k \in \N$ and telescoping this identity gives
\[
\frac{A_{n}}{B_{n}}=b_0+\sum_{k=0}^{n-1}\frac{(-1)^k\Pi_{k+1}}{B_{k}B_{k+1}}\,.
\]
Thereafter supposing, say
\begin{equation}\label{supposing}
\frac{\Pi_{k+1}}{B_{k}B_{k+1}} \longrightarrow 0\,,
\end{equation}
we have a limit
\[
\tau:=b_0+\sum_{k=0}^\infty\frac{(-1)^k\Pi_{k+1}}{B_{k}B_{k+1}}
\]
and consequently
\[
\tau-\frac{A_{n}}{B_{n}}=\sum_{k=n}^\infty\frac{(-1)^k\Pi_{k+1}}{B_{k}B_{k+1}}.
\]
Now, since the partial coefficients were assumed to be positive integers, we get
\[
\frac{\Pi_{n+1}}{B_{n}B_{n+1}}-\frac{\Pi_{n+2}}{B_{n+1}B_{n+2}}
<\left|\tau-\frac{A_{n}}{B_{n}}\right|
< \frac{\Pi_{n+1}}{B_{n}B_{n+1}}
\]
by the properties of alternating series.
Hence we have 
\begin{equation}\label{errorbound}
0<\frac{b_{n+2}\Pi_{n+1}}{B_{n+2}} < R_n < \frac{\Pi_{n+1}}{B_{n+1}}
\end{equation}
which is fundamental for our further studies. The above use of alternating series technique
is greatly inspired by Alf van der Poorten's lectures in \cite{alfvander}.

Sometimes we need different representations for the continued fraction under consideration. 
Then we may use for example the transformation
\begin{equation}\label{equivtransform}
b_0+\frac{a_1}{b_1} \alaplus \frac{a_2}{b_2} \alaplus \frac{a_3}{b_3}\alaplus \aladots=b_0+\frac{e_1a_1}{e_1b_1} \alaplus \frac{e_1e_2a_2}{e_2b_2} \alaplus \frac{e_2e_3a_3}{e_3b_3}\alaplus \aladots
\end{equation}
from \cite{perron}, provided that $e_n\ne 0$ for all $n \in \Z^+$ (this relation does not require $(a_n)$ and $(b_n)$ to be integer sequences).

\begin{Lemma}\label{INVARIANT}
Let $\tau \in \R$ and $q,t \in \Z \setminus \{0\}$, $r \in \Z$.
If there exist positive numbers $c_1$, $H_1$ and $\omega$, not depending on $N$, such that
\begin{equation}\label{LOWERBOUND1}
\lvert N\tau-M\rvert \ge \frac{c_1}{N^\omega}
\end{equation}
for all $M,N\in\Z$, $N \ge H_1$, then there exist positive numbers 
$c_i=c_i(\tau,\omega,c_1)$ and $H_i=H_i(\tau,H_1)$, $i\in\{2,3\}$, such that
\begin{equation}\label{LOWERBOUND2}
\left|N\left(\frac{q}{t}\tau+\frac{r}{t}\right)-M\right| \ge \frac{c_2}{N^\omega}
\end{equation}
for all $M,N\in\Z$, $N\ge H_2$ and
\begin{equation}\label{LOWERBOUND3}
\left|N\frac{1}{\tau}-M\right| \ge \frac{c_3}{N^\omega}
\end{equation}
for all $M,N\in\Z$, $N\ge H_3$.
\end{Lemma}
\begin{proof} 
Using \eqref{LOWERBOUND1} we get
\[
\left|N\left(\frac{q}{t}\tau+\frac{r}{t}\right)-M\right|=\frac{1}{\lvert t\rvert}\left|(Nq)\tau-(tM-rN)\right|
\ge \frac{c_1}{\lvert t\rvert\lvert q\rvert^\omega N^\omega} \,,
\]
when $N\lvert q\rvert\ge H_1$. This proves \eqref{LOWERBOUND2} with $c_2=\dfrac{c_1}{\lvert t\rvert\lvert q\rvert^\omega}$, 
$H_2=\dfrac{H_1}{\lvert q\rvert}$.
We move on to proving \eqref{LOWERBOUND3}.
Again by using the estimate \eqref{LOWERBOUND1} we get
\[
\left|N\frac{1}{\tau}-M\right|=
\frac{1}{\lvert\tau\rvert}\lvert M\tau-N\rvert
\ge \frac{c_1}{\lvert\tau\rvert\lvert M\rvert^\omega}\,,
\]
when $\lvert M\rvert\ge H_1$. It is enough to restrict to the case
\begin{equation}\label{RESTRICT}
\left|N\frac{1}{\tau}-M\right|\le 1 \qquad \Rightarrow \qquad
N\frac{1}{\lvert\tau\rvert}-1\le\lvert M\rvert\le N\frac{1}{\lvert \tau\rvert}+1\,.
\end{equation}
We require
\[
N\frac{1}{\lvert \tau\rvert}-1\ge H_1\,, 
\]
which quarantees that $\lvert M \rvert \ge H_1$.
Now \eqref{RESTRICT} implies
\[
\lvert M\rvert^\omega \le \left(N\frac{1}{\lvert\tau\rvert}+1\right)^\omega \le
N^\omega\left(\frac{1}{\lvert\tau\rvert}+1\right)^\omega\,.
\]
Thus
\[
\left|N\frac{1}{\tau}-M\right| \ge
\frac{c_1}{\lvert\tau\rvert\left(\frac{1}{\lvert\tau\rvert}+1\right)^\omega N^\omega} \,,
\]
and so we may put $c_3=\dfrac{c_1}{\lvert\tau\rvert\left(\frac{1}{\lvert\tau\rvert}+1\right)^\omega}$ and 
$H_3=\lvert\tau\rvert(H_1+1)$.
\end{proof}

\begin{Remark}
Lemma \ref{INVARIANT} tells that an irrationality exponent is invariant under
any fractional linear transformation, namely
\begin{equation}\label{FLTR}
\mu_I\left(\frac{r\tau+s}{t\tau+v}\right)=\mu_I(\tau)
\end{equation}
whenever $r,t,s,v\in\Z$ and \, $\displaystyle\frac{r\tau+s}{t\tau+v}\notin\Q$. The invariance (\ref{FLTR}) follows from 
\begin{equation}\label{}
\frac{r\tau+s}{t\tau+v}=\frac{st-rv}{t(t\tau+v)}+\frac{r}{t}
\end{equation}
with Lemma \ref{INVARIANT}. Consequently, the asymptotic irrationality exponent is also invariant under the linear fractional transformations. It is thus enough to consider the tails of generalized continued fractions when determining the asymptotic irrationality exponent. 
\end{Remark}

The following lemma is what we use through the paper for
deducing an upper bound for the asymptotic irrationality exponent from growth properties of the partial coefficients.
It is merely an adaptation of a well-known method in Diophantine approximation to the context of generalized continued fractions.
\begin{Lemma} \label{Theorem3.1}
Let $(a_n)$ and $(b_n)$ be sequences of positive integers and let \linebreak
$\tau=\K_{n=1}^\infty\frac{a_n}{b_n}$.
Suppose
\[
\lim_{n\to\infty}\frac {\log B_{n+1}}{\log B_n}=1
\] 
and
\begin{equation}\label{LIMSUPnu}
\limsup_{n\to \infty}\frac{\log \Pi_{n}}{\log B_{n}} \leq \nu<1\,.
\end{equation}
Then $\tau$ is irrational and 
\begin{equation}\label{ASYPIRRMEASURE}
\mu_I(\tau) \le 2+\frac{\nu}{1-\nu}\,.
\end{equation}
\end{Lemma}
\begin{proof}
First we note that the assumption (\ref{LIMSUPnu}) implies \eqref{supposing}
which in turn confirms the convergence of $\K_{n=1}^\infty\frac{a_n}{b_n}$.
Further, by \eqref{errorbound} and \eqref{LIMSUPnu} we have $0<R_n \to 0$ 
and so $\tau$ is irrational. Now we are ready to prove \eqref{ASYPIRRMEASURE}.
Let $\varepsilon>0$. For each $a \in \Z$, $b \in \Z^+$ denote $D_n=A_nb-B_na$. Since $\displaystyle\lim_{n\to\infty}R_n=0 $ we may define integers $k=\max\{j\in\Z^+\mid R_j > \frac{1}{2b}\}$ and $l=\min\{j\in\Z^+\mid R_j \le \frac{1}{2b} \text{ and } D_j\neq0\}$,
provided that $b$ is so big that $R_1 >\frac{1}{2b}$. Since
\[
A_nB_{n+1}-B_nA_{n+1}=(-1)^{n+1}\Pi_{n+1}\neq 0
\]
for all $n\in\Z^+$ we have $l=k+1$ or $l=k+2$.
Since $k\to\infty$ when $b\to\infty$ there exists $b_0\in\Z^+$ such that
\[
\frac{\log\Pi_{k+1}}{\log B_{k+1}}< \nu+\varepsilon
\]
and
\[
\frac{\log B_{k+2}}{\log B_{k+1}}< 1+\varepsilon
\]
whenever $b\ge b_0$. By \eqref{errorbound} now
\[
\frac{\log B_{k+2}}{-\log R_k}\le\frac{\log B_{k+2}}{\log B_{k+1}}\frac{\log B_{k+1}}{\log B_{k+1}-\log\Pi_{k+1}}<(1+\varepsilon)\frac{1}{1-(\nu+\varepsilon)}\,.
\]
Then
\[
1\le \lvert D_l\rvert=\lvert A_lb-B_la\rvert=\lvert B_l(b\tau-a)-(B_l\tau-A_l)b\rvert \le B_l\lvert b\tau-a\rvert +bR_l\,.
\]
By the choice of $l$ we have $bR_l \le\frac{1}{2}$ so $B_l\lvert b\tau-a\rvert >\frac{1}{2}$. Since $(B_n)$ is increasing we have
\[
B_l\le B_{k+2}<R_k^{-\frac{1+\varepsilon}{1-(\nu+\varepsilon)}}\leq (2b)^{\frac{1+\varepsilon}{1-(\nu+\varepsilon)}}\,.
\]
Now
\[
\left|\tau-\frac{a}{b}\right|=\frac{\lvert b\tau-a\rvert}{b}>\frac{1}{2bB_l}>\frac{1}{(2b)^{1+\frac{1+\varepsilon}{1-(\nu+\varepsilon)}}}
\]
for all $b\ge b_0$. Hence the asymptotic irrationality exponent of number $\tau$ is at most $1+\frac{1}{1-\nu}=2+\frac{\nu}{1-\nu}$.
\end{proof}

\section{Bounded partial coefficients} 

\begin{Theorem} \label{Theorem3.2}
Let $(a_n)$ and $(b_n)$ be sequences of positive integers and let
$\tau=\K_{n=1}^\infty\frac{a_n}{b_n}$.
Suppose
\[
\alpha_1\leq a_n\leq \alpha_2\,,\qquad\beta_1\leq b_n\leq\beta_2
\] 
for all $n\in\Z^+$, where
$\alpha_1$, $\alpha_2$, $\beta_1$ and $\beta_2$ are positive integers satisfying 
\begin{equation}\label{gamma1condition}
\gamma_1:=\frac{\beta_1+\sqrt{\beta_1^2+4\alpha_1}}{2}>\alpha_2\,.
\end{equation}  
Then $\tau$ is irrational and
\[
\mu_I(\tau)\le 2+\frac{\log\alpha_2}{\log\gamma_1-\log\alpha_2}\,.
\] 
\end{Theorem}
\begin{proof}
Since $(B_n)$ is an increasing sequence, 
\begin{align*}
1&<\frac{\log B_{n+1}}{\log B_n}=\frac{\log (b_{n+1}B_n+a_{n+1}B_{n-1})}{\log B_n}<\frac{\log ((\beta_2+\alpha_2)B_n)}{\log B_n}\\
&=\frac{\log(\beta_2+\alpha_2)}{\log B_n}+1\longrightarrow 1
\end{align*}
so $\displaystyle\lim_{n\to\infty}\frac{\log B_{n+1}}{\log B_n}=1$.
From
\[
B_{n+2}=b_{n+2}B_{n+1}+a_{n+2}B_n \ge \beta_1B_{n+1}+\alpha_1 B_n\,,
\] 
we get by solving the recurrence $B'_{n+2}=\beta_1B'_{n+1}+\alpha_1 B'_n$, where $B'_0=1$ and $B'_1=\beta_1$, that
\[
B_n\geq B'_n=\gamma_1^n\sum_{k=0}^n\left(\frac{\gamma_2}{\gamma_1}\right)^k=\gamma_1^nS_n\,,
\]
where   
$\gamma_2=\frac{\beta_1-\sqrt{\beta_1^2+4\alpha_1}}{2}$ and $\underset{n\to\infty}{\lim}S_n=\sum_{k=0}^{\infty}\left(\frac{\gamma_2}{\gamma_1}\right)^k$ converges. 
Now we get
\begin{align*}
\frac{\log\Pi_n}{\log B_n}\le\frac{n\log\alpha_2}{n\log \gamma_1+\log S_n}\longrightarrow\frac{\log\alpha_2}{\log \gamma_1}
\end{align*}
when $n\to\infty$. Hence the result follows from Lemma \ref{Theorem3.1}.
\end{proof}

\begin{Remark}
Theorem \ref{Theorem3.2} restricted to simple continued fractions gives $\mu_I(\tau)=2$ since $\alpha_2=1$. 
This is of course consistent with the basics of simple continued fraction theory
which tells that $2$ is even an irrationality exponent of simple continued fractions with bounded partial denominators (the so-called badly approximable numbers).
\end{Remark}

\begin{Remark}
Some condition like \eqref{gamma1condition} in Theorem \ref{Theorem3.2} is certainly necessary as the following theorem shows.
\end{Remark}

\begin{Theorem}\label{Theorem3.5}
For any $s \in \{1\} \cup [2,\infty]$ there exists a generalized continued fraction $\tau_s=\K_{n=1}^\infty\frac{a_n}{1}$ with
$a_n \in \{1,2\}$ for all $n \in \Z^+$ and $\mu_I(\tau_s)=s$.
\end{Theorem}
\begin{proof}
The case $s=1$ is dealt with by noting that $\K_{n=1}^\infty \frac{2}{1}=1$.\\
Next we calculate
\begin{align} 
\frac{2}{1} \alaplus \frac{2}{1} \alaplus \frac{1}{1} \alaplus
\frac{1}{c} \alaplus \frac{1}{1} \alaplus \frac{1}{1} \alaplus
\frac{2}{1} \alaplus \frac{2}{1} \alaplus \frac{x}{1}
& = \frac{(4c+5)x+8c+9}{(4c+6)x+8c+11} \nonumber \\
& = \frac{1}{1} \alaplus \frac{1}{4c+4} \alaplus \frac{1}{1} \alaplus
\frac{1}{1} \alaplus \frac{x}{1} \label{lasku}
\end{align}
for all positive integers $c$ and real numbers $x$. Then we show by induction that
\begin{align}
&\overbrace{\frac{2}{1} \alaplus \aladots \alaplus \frac{2}{1}}^{2k \text{ times}} \alaplus \frac{1}{1} \alaplus
\frac{1}{1} \alaplus \frac{1}{1} \alaplus \frac{1}{1} \alaplus
\overbrace{\frac{2}{1} \alaplus \aladots \alaplus \frac{2}{1}}^{2k \text{ times}}
\alaplus \frac{x}{1} \nonumber\\
=\,& \frac{1}{1} \alaplus \frac{1}{(7\cdot 4^{k}-4)/3} \alaplus \frac{1}{1} \alaplus \frac{1}{1}\alaplus \frac{x}{1}
\label{gentosimple}
\end{align}
for all $k\in\Z^+$ and real numbers $x$. The case $k=1$ follows immediately from \eqref{lasku} with $c=1$.
Let us assume that 
\begin{align}
&\overbrace{\frac{2}{1} \alaplus \aladots \alaplus \frac{2}{1}}^{2(k-1) \text{ times}} 
\alaplus \frac{1}{1} \alaplus 
\frac{1}{1} \alaplus \frac{1}{1} \alaplus \frac{1}{1} \alaplus
\overbrace{\frac{2}{1} \alaplus \aladots \alaplus \frac{2}{1}}^{2(k-1) \text{ times}}
 \alaplus \frac{x}{1} \nonumber\\
=\,& \frac{1}{1} \alaplus \frac{1}{(7\cdot 4^{k-1}-4)/3} \alaplus \frac{1}{1} \alaplus \frac{1}{1} \alaplus \frac{x}{1}
\label{io}
\end{align}
Assumption \eqref{io} holds especially when we substitute $\displaystyle \frac{x}{1}$ with 
$\displaystyle \frac{2}{1}\alaplus\frac{2}{1}\alaplus \frac{x}{1}$. Then we have
\begin{align*}
&\overbrace{\frac{2}{1} \alaplus \aladots \alaplus \frac{2}{1}}^{2k \text{ times}} 
\alaplus \frac{1}{1} \alaplus \frac{1}{1} \alaplus \frac{1}{1} \alaplus \frac{1}{1} \alaplus
\overbrace{\frac{2}{1} \alaplus \aladots \alaplus \frac{2}{1}}^{2k \text{ times}}
\alaplus \frac{x}{1}\\
=\,&\frac{2}{1}\alaplus\frac{2}{1}\alaplus \overbrace{\frac{2}{1} \alaplus \aladots \alaplus \frac{2}{1}}^{2(k-1) \text{ times}} \alaplus \frac{1}{1} \alaplus
\frac{1}{1} \alaplus \frac{1}{1} \alaplus \frac{1}{1} \alaplus
\overbrace{\frac{2}{1} \alaplus \aladots \alaplus \frac{2}{1}}^{2(k-1) \text{ times}}\alaplus\frac{2}{1}\alaplus\frac{2}{1}\alaplus\frac{x}{1}\\
=\,&\frac{2}{1}\alaplus\frac{2}{1}\alaplus \frac{1}{1} \alaplus \frac{1}{(7\cdot 4^{k-1}-4)/3} \alaplus \frac{1}{1} \alaplus
\frac{1}{1}\alaplus\frac{2}{1}\alaplus\frac{2}{1}\alaplus\frac{x}{1}\\
=\,& \frac{1}{1} \alaplus \frac{1}{4(7\cdot 4^{k-1}-4)/3+4} \alaplus \frac{1}{1} \alaplus
\frac{1}{1}\alaplus\frac{x}{1}=\frac{1}{1} \alaplus \frac{1}{(7\cdot 4^{k}-4)/3} \alaplus \frac{1}{1} \alaplus
\frac{1}{1}\alaplus\frac{x}{1} \,.  
\end{align*}
Hence (\ref{gentosimple}) holds for all real numbers $x$. \\
Now let $s \ge 2$ be a real number. 
Consider the simple continued fraction $\tau_s=[0,c_1,c_2,\ldots]$,
\begin{equation}\label{bnmaar}
c_n = 
\begin{cases}
1\,, \quad & \text{ when } n \not \equiv 2 \mod 4\,, \\
\frac{7\cdot4^{f(n)}-4}{3}\,, \quad & \text{ when } n \equiv 2 \mod 4\,,
\end{cases}
\end{equation}
where $f(n)$ is defined recursively as
\begin{equation}\label{fnmaar}
f(n)=\left\lceil \log_4\left(\frac{3B_{n-1}^{s-2}+4}{7}\right)\right\rceil\,.
\end{equation} 
By using formula \eqref{gentosimple} repeatedly we get 
\begin{align*}
\tau_s=&\,\frac{1}{1} \alaplus \frac{1}{(7\cdot 4^{f(2)}-4)/3} \alaplus \frac{1}{1} \alaplus \frac{1}{1} \alaplus \frac{1}{1} \alaplus \frac{1}{(7\cdot 4^{f(6)}-4)/3} \alaplus \frac{1}{1} \alaplus \frac{1}{1}\alaplus\aladots \\
=&\,\overbrace{\frac{2}{1} \alaplus \aladots \alaplus \frac{2}{1}}^{2f(2) \text{ times}} \alaplus \frac{1}{1} \alaplus \frac{1}{1} \alaplus \frac{1}{1} \alaplus \frac{1}{1} \alaplus \overbrace{\frac{2}{1} \alaplus \aladots \alaplus \frac{2}{1}}^{2f(2) \text{ times}}\\ 
&\qquad\alaplus \overbrace{\frac{2}{1} \alaplus \aladots \alaplus \frac{2}{1}}^{2f(6) \text{ times}} \alaplus \frac{1}{1} \alaplus \frac{1}{1} \alaplus \frac{1}{1} \alaplus \frac{1}{1} \alaplus \overbrace{\frac{2}{1} \alaplus \aladots \alaplus \frac{2}{1}}^{2f(6) \text{ times}}\alaplus \overbrace{\frac{2}{1} \alaplus \aladots \alaplus \frac{2}{1}}^{2f(10) \text{ times}}\alaplus\aladots\\
=&\,\K_{n=1}^\infty \frac{a_n}{1}
\end{align*}
with $a_n \in \{1,2\}$ for each $n \in \Z^+$.
From \eqref{bnmaar} and \eqref{fnmaar} we have $c_n\ge B_{n-1}^{s-2}$ for all $n\equiv 2 \mod 4$ ($B_{n-1}$ still refers to the denominator of the simple continued fraction $[0,c_1,c_2,\ldots]$ of $\tau_s$, not the generalized one). Hence
\[
\left|\tau_s-\frac{A_n}{B_n}\right| < \frac{1}{b_{n+1}B_n^2}\le \frac{1}{B_n^s}
\]
for all $n\equiv 1 \mod 4$ and thus $\mu_I(\tau_s) \ge s$. On the other hand, when $s>2$ \eqref{bnmaar} and \eqref{fnmaar} imply
$c_n<5B_{n-1}^{s-2}-2$ for large enough $n$ and then
\[
\left|\tau_s-\frac{A_n}{B_n}\right| > \frac{1}{(b_{n+1}+2)B_n^2}>\frac{1}{5B_n^s}\,.
\]
Only the convergents of the simple continued fraction expansion may violate \eqref{mudefinition} when $\mu \ge 2$ and $c\le \frac{1}{2}$, so we get $\mu_I(\tau_s) \le s$ also.
When we define $\tau_{\infty}$ by substituting $s$ with $n$ in \eqref{fnmaar}, we get
\[
\left|\tau_{\infty}-\frac{A_n}{B_n}\right|<\frac{1}{B_n^n}
\]
for all $n\equiv 1 \mod 4$ and thus $\tau_{\infty}$ is a Liouville number.
\end{proof}

\begin{Example}
With Lemma \ref{Theorem3.1} and Theorem \ref{Theorem3.2} we may consider for example some continued fractions generated by finite automata.
The celebrated Thue-Morse $\overline{t}=t_1t_2\ldots$, $t_i\in\{0,1\}$ and 
Fibonacci $\overline{f}=f_1f_2\ldots$, $f_i\in\{0,1\}$ 
infinite words are defined as fixed points of binary morphisms
\begin{align*}
&\sigma_t(0)=01 \,, \quad \sigma_t(1)=10 \,, \quad &\overline{t}&=\sigma_t^{\infty}(0)=01101001100 \ldots\,,\\
&\sigma_f(0)=01 \,, \quad \sigma_f(1)=0  \,, \quad &\overline{f}&=\sigma_f^{\infty}(0)=01001010010 \ldots\,,
\end{align*}
respectively. 
The infinite words $\overline{t}=t_1t_2\ldots$ and
$\overline{f}=f_1f_2\ldots$ generate the Thue-Morse $(t_n)=(0,1,1,0,\ldots)$ and 
Fibonacci  $(f_n)=(0,1,0,0,\ldots)$
binary sequences, which have attracted a great deal of attention in several contexts, see \cite{adabug2007} and \cite{adabug2005}.
We will study the following continued fractions
\[
\tau_t=\K_{n=1}^\infty \frac{2^{t_n}}{2^{t_n+1}}\,, \quad
\tau_f=\K_{n=1}^\infty \frac{2^{f_n}}{2^{f_n+1}}\,, \quad
\tau_{ft}=\K_{n=1}^\infty \frac{2^{f_n}}{2^{t_n}}\,.
\]
First we note, that by Theorem \ref{Theorem3.2} we get
\begin{equation}\label{badresults}
\mu_I(\tau_t),\ \mu_I(\tau_f) \le 2+\frac{\log 2}{\log(1+\sqrt{2})-\log 2}= 5.682\ldots\,,
\end{equation}
but $\mu_I(\tau_{ft})$ is out of reach. \\
However, when we know some density properties of the sequences $(a_n)$ and $(b_n)$
we may do better. It is known and quite straightforward to establish that 
\begin{equation}\label{thueden}
\lim_{n\to\infty} \frac{\#\{k\mid t_k=1,\,1\leq k\leq n\}}{n}=\frac{1}{2}
\end{equation}
and from \cite{lothaire} we have
\begin{equation}\label{fiboden}
\lim_{n\to\infty} \frac{\#\{k\mid f_k=1,\,1\leq k\leq n\}}{n}=\frac{1}{\varphi^2}\,,
\end{equation}
where $\varphi=\frac{1+\sqrt{5}}{2}$ is the golden ratio. With the help of these formulas a direct use of Lemma \ref{Theorem3.1} gives better results. \\
Let $\varepsilon>0$ be arbitrary. For big enough $n$ we now have
\[
\Pi_n \le \big(\sqrt{2}\,(1+\varepsilon)\big)^n
\]
by \eqref{thueden}.
From the estimate $B_{n+2}\ge 2B_{n+1}+B_n$ we deduce via solving the recurrence that $B_{n+1}\ge (1+\sqrt{2}-\varepsilon)B_n$ for all large $n$. Thus from some point on,
\[
B_n =2^{t_n}\left(2+\frac{1}{2^{t_{n-1}}\left( 2+\frac{B_{n-3}}{B_{n-2}}\right)}\right)B_{n-1} 
\ge 2^{t_n}\left(2+\frac{\sqrt{2}-1-\varepsilon}{2^{t_{n-1}}}\right) B_{n-1}\,.
\]
Repeating this $n$ times and using \eqref{thueden} gives the bound
\[
B_n \ge \big(\sqrt{5+4\sqrt{2}}\,(1-\varepsilon)\big)^n\,.
\]
Now we get to use Lemma \ref{Theorem3.1} with
\[
\frac{\log \Pi_n}{\log B_n} \le \frac{\log\sqrt{2}+\log(1+\varepsilon)}{\log\sqrt{5+4\sqrt{2}}+\log(1-\varepsilon)} 
\longrightarrow \frac{\log 2}{\log(5+4\sqrt{2})}
\]
implying
\[
\mu_I(\tau_t)\le 2+ \frac{\log{2}}{\log(5+4\sqrt{2})-\log{2}}= 2.414\ldots\,,
\]
which beats \eqref{badresults}. If we take in to account that the asymptotic proportions of the subwords $00$, $01$, $10$ and $11$ in the Thue-Morse word are $\frac{1}{6}$, $\frac{1}{3}$, $\frac{1}{3}$ and $\frac{1}{6}$, respectively, we can further improve this bound to $2.412\ldots$, and so on. \\
In the next case we use the same arguments with \eqref{fiboden} to bound
\[
\Pi_n \le \big(2^{\frac{1}{\varphi^2}}\,(1+\varepsilon)\big)^n
\]
and
\[
B_n \ge \big( (1+\sqrt{2})^{\frac{1}{\varphi}}(3+\sqrt{2})^{\frac{1}{\varphi^2}}\,(1-\varepsilon)\big)^n
\]
for all $\varepsilon>0$ and $n$ sufficiently large. With these bounds, Lemma \ref{Theorem3.1} implies
\[
\mu_I(\tau_f) \le 2+ \frac{\log 2}{\varphi\log(1+\sqrt 2)+\log(3+\sqrt 2)-\log 2}= 2.312\ldots\,,
\]
again beating \eqref{badresults}. Again, after noting that the asymptotic proportions of the subwords $00$, $01$ and $10$ of the Fibonacci word are $1-\frac{2}{\varphi^2}$, $\frac{1}{\varphi^2}$ and $\frac{1}{\varphi^2}$, respectively, (follows immediately from \eqref{fiboden} since the subword $11$ never occurs) we get an improvement to $2.311\ldots$, and so on. \\
In the final case we may simply estimate
\[
\Pi_n\le\big(2^{\frac{1}{\varphi^2}}\,(1+\varepsilon)\big)^n
\]
and
\[
B_n\ge \bigg(\sqrt{\frac{3+\sqrt{6}}{2}}\,(1-\varepsilon)\bigg)^n
\]
to get
\[
\mu_I(\tau_{ft}) \le 2+\frac{\log{4}}{\varphi^2(\log(3+\sqrt{6})-\log{2})-\log{4}}=3.119\ldots\,.
\]
\end{Example}

\section{Polynomial growth}

\begin{Theorem}\label{poly}
Let $(a_n)$ and $(b_n)$ be sequences of positive integers and let $\tau =\K_{n=1}^\infty\frac{a_n}{b_n}$. 
Suppose
\[
a_n \le \alpha n^l\,,\qquad \beta_1 n^{k_1}\le b_n \le \beta_2 n^{k_2}
\]
for all $n \in \Z^+$, where $l$ is a non-negative real number and $\beta_1$, $\beta_2$, $k_1$ and $k_2$ are positive real numbers satisfying
\begin{equation}\label{polycondition}
l<k_1 \leq k_2\,.
\end{equation}
Then $\tau$ is irrational and 
\[
\mu_I(\tau)\le 2+\frac{l}{k_1-l}\,.
\] 
\end{Theorem}
\begin{proof}
Since
\[
B_n>b_nB_{n-1}\ge\beta_1 n^{k_1}B_{n-1}>\beta_1^n(n!)^{k_1}
\]
we have
\begin{align*}
1&<\frac{\log B_{n+1}}{\log B_n}=\frac{\log (b_{n+1}B_n+a_{n+1}B_{n-1})}{\log B_n}<\frac{\log ((\beta_2n^{k_2}+\alpha n^l)B_n)}{\log B_n}\\
&<\frac{\log(\beta_2+\alpha)+k_2\log n}{n\log\beta_1+k_1\log n!}+1\longrightarrow 1
\end{align*}
so $\displaystyle\lim_{n\to\infty}\frac{\log B_{n+1}}{\log B_n}=1$. Further
\begin{align*}
\frac{\log\Pi_n}{\log B_n}&\le\frac{\log(\alpha^n(n!)^l}{\log(\beta_1^n(n!)^{k_1}}=\frac{n\log\alpha+l\log(n!)}{n\log\beta_1+
k_1\log(n!)}\longrightarrow\frac{l}{k_1}
\end{align*}
when $n\to\infty$. Hence the result follows from Lemma \ref{Theorem3.1}.
\end{proof}

\begin{Remark}
Note that whenever $p(x)$ is a polynomial of degree $d$ with a positive leading coefficient,
there exist positive constants $c_1$ and $c_2$ such that $c_1n^d \le p(n) \le c_2n^d$ for large $n$. So Theorem \ref{poly} is suitable for cases where partial coefficients are values of polynomials or bounded by them.
\end{Remark}

\begin{Remark}
Some condition like \eqref{polycondition} in Theorem \ref{poly} is again necessary as can be seen from 
example \eqref{rationalexample} of Bowman and Mc Laughlin.
\end{Remark}

\begin{Example} Let us consider the exponential function in rational points. Let $x\in\Z\setminus\{0\}$ and $y\in\Z^+$. Then $e^{\frac{x}{y}}$ has continued fraction expansion (see for example \cite{lorentzen})
\[
e^{\frac{x}{y}}=1+\frac{2x}{2y-x}\alaplus\frac{x^2}{6y}\alaplus\frac{x^2}{10y}\alaplus\frac{x^2}{14y}\alaplus\aladots=1+\cfrac{2x}{2y-x+\tau}\,.
\]
Due to Lemma \ref{INVARIANT} we may consider only the tail part $\tau$.
Now with $l=0$ and $k_1=k_2=1$ Theorem \ref{poly} gives us $\mu_I(\tau)\le 2$ and hence $\mu_I(e^{\frac{x}{y}})=2$. This result has been previously obtained by for example Shiokawa in \cite{shiokawa}.
\end{Example}

\begin{Example} 
Next we shall consider generalizations of the type of simple continued fractions studied by Bundschuh in \cite{bundschuh1998}. 
Let
\[
\tau = \K _{n=1}^\infty \frac{1}{c_n+d_n\left\lceil\frac{n}{s}\right\rceil^m}
\]
with
\[
c_n = \frac{t_{\overline{n}}}{u_{\overline{n}}} \quad \text {and} \quad
d_n = \frac{v_{\overline{n}}}{w_{\overline{n}}}\,,\quad \text{when}\quad n \equiv \overline{n} \mod s\,,\quad 0<\overline{n}\le s\,,
\]
where $m,s,t_i,u_i,v_i,w_i \in \Z^+$ for $i=1,2,\ldots,s$. Bundschuh has $u_i=w_i=1$ for $i=1,2,\ldots,s$.
Using the equivalence transformation \eqref{equivtransform} with 
$e_n=u_{\overline{n}}w_{\overline{n}}$ 
we get the generalized continued fraction representation $\tau=\K_{n=1}^\infty \frac{a_n}{b_n}$ with
positive integer partial numerators
$a_1=u_1w_1$ and
$a_n=u_{\overline{n-1}}\,w_{\overline{n-1}}\,u_{\overline{n}}\,w_{\overline{n}}$ for $n\in\Z^+$, $n>1$ and denominators
$b_n=t_{\overline{n}}w_{\overline{n}}+u_{\overline{n}}v_{\overline{n}}\lceil\frac{n}{s}\rceil^m $ for $n\in \Z^+$.
Now we have
\[
1 \le a_n\le \underset{1\le \overline{n}\le s}
\max\{u_{\overline{n}}\,w_{\overline{n}}\,u_{\overline{n+1}}\,w_{\overline{n+1}}\}
\]
and
\[
\frac{n^m}{s^m} \le b_n\le \underset{1\le \overline{n}\le s}
\max\{t_{\overline{n}}\,w_{\overline{n}}+v_{\overline{n}}\,u_{\overline{n}}\}\frac{(1+s)^mn^m}{s^m}
\]
meaning that $\mu_I(\tau)=2$ by Theorem \ref{poly} with $l=0$ and $k_1=k_2=m$. 
\end{Example}

\section{Exponential growth}

\begin{Theorem}\label{exp}
Let $(a_n)$ and $(b_n)$ be sequences of positive integers and let $\tau=\K_{n=1}^\infty\frac{a_n}{b_n}$.
Suppose 
\[
a_n\le r\alpha^{n^l}\,, \qquad s_1\beta_1^{n^{k_1}}\le b_n\le s_2\beta_2^{n^{k_2}}
\]
for all $n\in\Z^+$, where $l$ is a non-negative real number and
$r$, $s_1$, $s_2$, $\alpha$, $\beta_1$, $\beta_2$, $k_1$ and $k_2$ are positive real numbers satisfying $k_1+1>k_2\ge k_1$.
\begin{enumerate}
\item If $l<k_1$ then $\tau$ is irrational and $\mu_I(\tau)=2$.
\item If $l=k_1$ and $\alpha<\beta_1$ then $\tau$ is irrational and 
\[
\mu_I(\tau)\le 2+\frac{\log\alpha}{\log\beta_1-\log\alpha}\,.
\]
\end{enumerate} 
\end{Theorem}
\begin{proof}
Let us denote $\Sigma_h(n)=1^h+2^h+\cdots+n^h$. By integrating we get
\[
\frac{n^{h+1}}{h+1} < \Sigma_h(n) < \frac{(n+1)^{h+1}}{h+1}\,.
\]
Then
\[
B_n>b_nB_{n-1}\ge s_1\beta_1^{n^{k_1}}B_{n-1}>
\cdots> s_1^n\beta_1^{\Sigma_{k_1}(n)}
\]
and
\begin{align*}
1& < \frac{\log B_{n+1}}{\log B_n}=\frac{\log (b_{n+1}B_n+a_{n+1}B_{n-1})}{\log B_n}
<\frac{\log ((s_2\beta_2^{(n+1)^{k_2}}+r\alpha^{(n+1)^l})B_n)}{\log B_n}\\
&\le \frac{\log(s_2+r)+(n+1)^{k_2}\log (\max\{\alpha,\beta_2\})}{n\log s_1+\Sigma_{k_1}(n)\log \beta_1}+1 \\
& < \frac{\log(s_2+r)+(n+1)^{k_2}\log (\max\{\alpha,\beta_2\})}{n\log s_1+\frac{n^{k_1+1}}{k_1+1}\log \beta_1}+1 
\longrightarrow 1
\end{align*}
so $\displaystyle\lim_{n\to\infty}\frac{\log B_{n+1}}{\log B_n}=1$. Further
\begin{align*}
\frac{\log\Pi_n}{\log B_n}&\le\frac{\log(r^n\alpha^{\Sigma_l(n)})}{\log(s_1^n\beta_1^{\Sigma_{k_1}(n)})}=\frac{n\log r+\Sigma_l(n)\log\alpha}{n\log s_1+\Sigma_{k_1}(n)\log\beta_1} \\
& \longrightarrow
\begin{cases}
0\,,\quad & \text{ when } l < k_1\,, \\
\frac{\log\alpha}{\log \beta_1}\,, \quad & \text{ when } l=k_1\,. 
\end{cases}
\end{align*}
Hence the result follows from Lemma \ref{Theorem3.1}.
\end{proof}

\begin{Example}\label{rogersram}
Let us consider the Rogers-Ramanujan continued fraction
\[
RR(q,t)=\frac{qt}{1}\alaplus\frac{q^2t}{1}\alaplus\frac{q^3t}{1}\alaplus\aladots\,,
\]
where $q=\frac{a}{b}$, $t=\frac{r}{s}$ and $a,b,r,s\in\Z^+$. Using transformation \eqref{equivtransform} with
\[
e_n=
\begin{cases}
b^{\frac{n+1}{2}}s\,\quad & \text{ when $n$ is odd,}\\
b^{\frac{n}{2}}\,, \quad & \text{ when $n$ is even}
\end{cases}
\]
we get
\begin{align*}
RR(q,t)&=\frac{qt}{1}\alaplus\frac{q^2t}{1}\alaplus\frac{q^3t}{1}\alaplus\frac{q^4t}{1}\alaplus\aladots=
\frac{\frac{a}{b}\frac{r}{s}}{1}\alaplus\frac{(\frac{a}{b})^2\frac{r}{s}}{1}\alaplus
\frac{(\frac{a}{b})^3\frac{r}{s}}{1}\alaplus\frac{(\frac{a}{b})^4\frac{r}{s}}{1}\alaplus\aladots \\
&=\frac{ar}{bs}\alaplus\frac{a^2r}{b}\alaplus\frac{a^3r}{b^2s}\alaplus\frac{a^4r}{b^2}\alaplus\aladots=
\K_{n=1}^\infty \frac{a_n}{b_n}
\end{align*}
where $a_n=a^nr$ and
\[
b_n=\begin{cases}
b^{\frac{n+1}{2}}s \,,\quad & \text{ when $n$ is odd,}\\
b^{\frac{n}{2}} \,, \quad & \text{ when $n$ is even.}
\end{cases}
\]
Hence $a_n\leq ra^n$ and $(\sqrt{b})^n\le b_n\le s\sqrt{b}(\sqrt{b})^n$. When $a^2<b$ Theorem \ref{exp} gives
\[
\mu_I(RR(q,t))\le 2+\frac{2\log a}{\log b-2\log a}\,.
\]
In particular, when $a=1$ we get $\mu_I(RR(\frac{1}{b},t))=2$ for all $b\in\Z^+$, $b \ge 2$,
a result proved already by  Bundschuh \cite{bundschuh1970}.
For more general approximation measures for $q$-continued fractions, see \cite{matmer}. 
In \cite{matmer} we may find an example 
\[
M(q)=\K_{n=1}^\infty\frac{q^{2n}}{1+q^n}
\]
where $q=\frac{a}{b}$, $a,b\in\Z^+$. Similarily to the previous example we get
\[
\mu_I(M(q))\le 2+\frac{2\log a}{\log b-2\log a}
\]
when $a^2<b$.
\end{Example}

\begin{Example} Next we shall consider Tasoev's continued fractions previously studied e.g.
in \cite{komatsu2003}, \cite{komatsu2005} and \cite{matmer}. We define
\begin{align*}
T_1(u,v,a)&=\frac{1}{ua}\alaplus\frac{1}{va^2}\alaplus\frac{1}{ua^3}\alaplus\frac{1}{va^4}\alaplus\aladots\,,\\
T_2(u,v,a,b)&=\frac{1}{ua}\alaplus\frac{1}{vb}\alaplus\frac{1}{ua^2}\alaplus\frac{1}{vb^2}\alaplus\frac{1}{ua^3}\alaplus\aladots\,,
\end{align*}
where $u$, $v$, $a$ and $b$ are positive rational numbers. Denote $a=\frac{x}{y}$ and $b=\frac{s}{t}$, $x,y,s,t \in \Z^+$. With a little calculation we get such representations for $T_1$ and $T_2$ that we can apply Theorem \ref{exp}. When $y^2<x$ we have
\[
\mu_I(T_1)\le 2+\frac{2\log y}{\log x-2\log y}\,.
\]
When $yt<\min\{x,s\}$ we have
\[
\mu_I(T_2)\le 2+\frac{\log y+\log t}{\log\min\{x,s\}-\log y-\log t}\,.
\]
\end{Example}

\begin{Example}\label{Tribonacci}
The Tribonacci sequence may be defined by the initial values $T_0=0, T_1=0, T_2=1$ and
the recursive formula $T_{n+3}=T_{n+2}+T_{n+1}+T_{n}$, $n\in\N$.
When $1\le l<k$ the continued fraction
\[
\tau=\K_{n=1}^\infty\frac{T_{n^l}}{T_{n^k}}
\]
has the asymptotic irrationality exponent $\mu_I(\tau)=2$ by Theorem \ref{exp}.
\end{Example}


\begin{thebibliography}{99}
\bibitem{adabug2007} B. Adamczewski and Y. Bugeaud,
On the complexity of algebraic numbers I. Expansions in integer bases, Ann. of Math. (2) 165 (2007), 547-565.
\bibitem{adabug2005} B. Adamczewski and Y. Bugeaud,
On the complexity of algebraic numbers II. Continued fractions, Acta Math. 195 (2005), 1-20.
\bibitem{bowman} D. Bowman and J. Mc Laughlin,
Polynomial Continued Fractions,  Acta Arith. 103 (2002), no. 4, 329-342.
\bibitem{bundschuh1970} P. Bundschuh, Ein Satz uber ganze Funktionen und Irrationalitatsaussagen, 
Invent. Math. 9 (1970) 175-184.
\bibitem{bundschuh1998} P. Bundschuh, On simple continued fractions with partial quotients in arithmetic progressions, 
Lith. Math. J. 38 (1) (1998), 15-26.
\bibitem{euler} L. Euler, De fractionibus continuis dissertation, Comm. Acad. Sci. Petr. 9, 98-137, (1737) 1744, 
Reprinted in Leonhardi Euleri Opera Omnia, Ser. I, Vol. 14. Leipzig, Germany: Teubner, pp. 187-215, 1924.
\bibitem{feldman} N. I. Fel'dman, Yu. V. Nesterenko, Transcendental Numbers, Number Theory IV, 
Encyclopaedia Math. Sci. 44, Springer, Berlin, 1998, pp. 1-345.
\bibitem{hanleilepmat}J. Han\v cl, M. Leinonen, K. Lepp\"al\"a, T. Matala-aho, 
Explicit irrationality measures for continued fractions, J. Number Theory 132 (8) (2012), 1758-1769. 
\bibitem{komatsu2003} T. Komatsu (2003). On Tasoev's continued fractions. 
Math. Proc. Cambridge Philos. Soc. 134 (2003), 1-12.
\bibitem{komatsu2005} T. Komatsu, Rational approximations to Tasoev continued fractions. II. 
Liet. Mat. Rink. 45 (2005), 84-94; translation in Lithuanian Math. J. 45 (2005), 66-75.
\bibitem{lehmer1973} D.H. Lehmer, Continued fractions containing arithmetic progressions, Scripta Math. 29 (1973) 17-24.
\bibitem{lorentzen} L. Lorentzen and H. Waadeland, Continued Fractions with Applications, Stud. Comput.
Math. 3, North-Holland Publishing Co., Amsterdam, 1992.
\bibitem{lothaire} M. Lothaire, Algebraic Combinatorics on Words, 
 Encyclopedia of Mathematics and its Applications 90, Cambridge University Press, Cambridge, 2002. 
\bibitem{matmer} T. Matala-aho, V. Meril\"a, On Diophantine approximations of Ramanujan type $q$-continued fractions, 
J. Number Theory 129 (5) (2009), 1044-1055. 
\bibitem{perron} O. Perron, Die Lehre von den Kettenbr\" uchen, B.G. Teubner, Leipzig, 1913. 
\bibitem{alfvander} A. J. van der Poorten, An introduction to continued fractions. Diophantine analysis (Kensington, 1985), 
London Math. Soc. Lecture Note Ser. 109, Cambridge Univ. Press, Cambridge, 1986, pp. 99-138.
\bibitem{shiokawa} I. Shiokawa, Rational approximations to the values of certain hypergeometric functions, 
in: Number Theory and Combinatorics, Japan, 1984 (Tokyo, Okayama and Kyoto, 1984), 
World Sci. Publishing, Singapore, 1985, pp. 353-367.
\bibitem{steuding} J. Steuding, Diophantine analysis,
Discrete Mathematics and its Applications (Boca Raton). Chapman \& Hall/CRC, Boca Raton, FL, 2005. viii+261 pp.
\end{thebibliography}
\end{document}